\documentclass[12pt]{amsart}
\usepackage{fullpage}
\usepackage{latexsym}
\usepackage{amssymb}
\usepackage[all]{xy}

\newtheorem{thmx}{Theorem}

\newtheorem{thm}{Theorem}[section]

\newtheorem{lem}[thm]{Lemma}
\newtheorem{que}[thm]{Question}
\newtheorem{prob}[thm]{Problem}

\newtheorem{rem}[thm]{Remark}
\newtheorem{ex}[thm]{Example}

\usepackage{hyperref}
 
\newcommand{\CP}{\mathbb{CP}}

\newcommand{\N}{\mathbb{N}}
\newcommand{\Q}{\mathbb{Q}}

\newcommand{\Z}{\mathbb{Z}}

\newcommand{\SO}{\mathrm{\SO}}

\title[On the realisation problem for mapping degree sets]
{On the realisation problem for mapping degree sets}
\author{Christoforos Neofytidis}
\address{Department of Mathematics, Ohio State University, Columbus, OH 43210, USA}
\email{neofytidis.1@osu.edu}
\author{Hongbin Sun}
\address{Department of Mathematics, Rutgers University - New Brunswick, Hill Center, Busch Campus, Piscataway, NJ 08854, USA}
\email{hongbin.sun@rutgers.edu}
\author{Ye Tian}
\address{Morningside Center of Mathematics, 
Academy of Mathematics and System Science, 
Chinese Academy of Sciences,
Beijing 100190}
\email{ytian@math.ac.cn}
\author{Shicheng Wang}
\address{Department of Mathematical Sciences, Peking University, Beijing 100871, CHINA}
\email{wangsc@math.pku.edu.cn}
\author{Zhongzi Wang}
\address{Department of Mathematical Sciences, Peking University, Beijing 100871 CHINA}
\email{wangzz22@stu.pku.edu.cn}
\date{\today}

\begin{document}

\maketitle

\begin{abstract} 
The set of degrees of maps $D(M,N)$, where $M,N$ are closed oriented $n$-manifolds, always contains 0 and the set  of degrees of self-maps $D(M)$ always contains $0$ and $1$. 
Also, if $a,b\in D(M)$, then $ab\in D(M)$; a set $A\subseteq\Z$ so that $ab\in A$ for each $a,b\in A$ is called multiplicative. On the one hand, not every infinite set of integers (containing $0$) is a mapping degree set~\cite{NWW} and, on the other hand, every finite set  of integers (containing $0$) is the mapping degree set of some $3$-manifolds~\cite{CMV}. We show the following:
\begin{itemize}
\item[(i)] Not every multiplicative set $A$ containing $0,1$ is a self-mapping degree set.
\item[(ii)] For each $n\in\N$ and $k\geq3$, every $D(M,N)$ for $n$-manifolds $M$ and $N$ is $D(P,Q)$ for some $(n+k)$-manifolds $P$ and $Q$.
\end{itemize}
As a consequence of (ii) and~\cite{CMV}, every finite set of integers (containing $0$) is the mapping degree set of some $n$-manifolds for all $n\neq 1,2,4,5$.
\end{abstract}
\section{Introduction}

Let $M,N$ be two closed oriented manifolds  of the same dimension.
The degree of a map $f\colon M\to N$, denoted by $\deg(f)$, is 
probably one of the  oldest and most fundamental concepts in topology.
The {\em set of degrees of maps} from $M$ to $N$ is defined by
\[
D(M,N):=\{d\in\Z \ | \ \exists \ f\colon M\to N, \ \deg(f)=d\}.
\]
 When $M=N$, the {\em set of degrees of self-maps} $D(M,M)$ is denoted by $D(M)$.

The following question, about realising subsets of integers as mapping degree sets, has been circulating for some years, but  formally appeared only recently:

\begin{prob}\label{problem1}\cite[Problem 1.1]{NWW}
Given a set $A\subseteq\Z$ with $0\in A$, are there closed oriented manifolds $M$ and $N$ such that $D(M,N)=A$?
\end{prob}

On the one hand, a negative answer has been given for infinite sets:

\begin{thmx}\label{thmA}\cite[Theorem 1.3]{NWW}
There exists an infinite set $A\subseteq\Z$ with $0\in A$ which is not $D(M,N)$, for any closed oriented $n$-manifolds $M,N$.
\end{thmx}

On the other hand, using 3-manifolds, which are connected sums of non-trivial circle bundles over hyperbolic surfaces, and their products, it was proved in~\cite[Theorems 1.7 and 1.9]{NWW} that many finite subsets of integers are mapping degree sets, including  finite arithmetic progressions (containing $0$)  and  finite geometric progressions of positive integers starting from 1 (and containing $0$). So, the following natural problems arose:

\begin{prob}\label{problem3}  \cite[Problem 1.3]{NWW}
Suppose $A$ is a finite set of integers containing $0$. Is $A=D(M,N)$, for some closed oriented $n$-manifolds $M$ and $N$?
\end{prob}

\begin{prob}\label{arith} \cite[Problem 1.6]{NWW}
Can every  arithmetic progression  containing $0$  be realised as $D(M,N)$, for some closed oriented $n$-manifolds $M,N$?
\end{prob}

\begin{prob}\label{geom} \cite[Problem 1.8]{NWW}
Together with $0$,  can every geometric progression of integers  be realised as $D(M,N)$, for some closed oriented $n$-manifolds $M,N$?
\end{prob}

Recently, Costoya, Mu\~noz and  Viruel gave a complete positive answer to Problem \ref{problem3} in a stronger form:
 
\begin{thmx}\label{t:CMV} \cite[Theorem A]{CMV}.
If $A$ is a finite set of integers containing $0$, then $A = D(M,N)$
for some closed oriented connected $3$-manifolds $M,N$.
\end{thmx}

It is  rather surprising that
all $3$-manifolds used in Theorem \ref{t:CMV} are just connected sums of non-trivial circle bundles over hyperbolic surfaces.
Further realisability results are shown in~\cite{CMV}, including that any finite set of integers containing $0$ is the mapping degree set of
some simply connected  $(4k-1)$-manifolds for $k>3$  \cite[Theorem C]{CMV}.

The self-mapping degree $D(M)$, a subset  of integers associated with a given closed oriented manifold $M$, 
 is  very interesting from a number theoretic  point of view. We say that a set $A\subseteq\Z$ is {\em multiplicative},  if $a, b \in A$ implies that $ab\in A$.
When $A$ is finite and multiplicative, then clearly $A\subseteq\{-1, 0, 1\}$.
Note that $D(M)$ is a multiplicative set containing $\{0, 1\}$:  $0$ is realised by a constant map, $1$ is realised by the identity, and  if $a, b\in D(M)$, then $ab\in D(M)$, which is realised by a self-map $g\circ f$ of $M$, where $g$ and $f$ are self-maps of $M$ realising $a$ and $b$ respectively.

Similarly to Problem \ref{problem1}, the following question has been also  circulating over the years:

\begin{prob}\label{problem2}
Suppose $A$ is a multiplicative set of integers containing $\{0, 1\}$.  Is there a closed oriented manifold $M$ with $D(M)=A$?
\end{prob}

It is worth mentioning that, 
although any finite mapping degree set is the mapping degree set of some $3$-manifolds, the corresponding statement is not true for infinite self-mapping degrees, as explained below: 

\begin{ex}\label{cpn}
Recall that $D(\CP^n)=\{k^n \ | \  k\in \Z\},$
where $\CP^n$ is the $n$-dimensional complex projective plane. 
However, for $n>2$,  one can check, with some number theoretic arguments, that $\{k^n \ | \  k\in \Z\}$ is not the self-mapping degree set $D(M)$ of any 3-manifold $M$, according to \cite{SWWZ}. 
\end{ex}

Although the proof of the claim $D(\CP^n)=\{k^n \ | \  k\in \Z\}$ in Example \ref{cpn} should be well-known, we will include it at the end of paper, since we could not locate a precise reference.

\section{Results and questions}

Our first result in this note is a negative answer to Problem \ref{problem2}:
 
\begin{thm}\label{non1}
There exists a multiplicative set $A$ containing $\{0, 1\}$ which is not $D(M)$ for any closed oriented $n$-manifold $M$. 
\end{thm} 

Since higher dimensional manifolds are richer than those in lower dimensions,  it is natural to expect  that $D(M, N)$, for any pair of $n$-manifolds $M,N$, is realised by 
a pair of higher dimensional manifolds. The next result supports this expectation:

\begin{thm}\label{inclusion}  
For each $n\in\N$ and $k\geq3$, every mapping degree set $A$ of $n$-manifolds is the mapping degree set of some $(n+k)$-manifolds. Moreover there are infinitely many pairs of $(n+k)$ manifolds realising $A$. The same is true for self-mapping degree sets.
\end{thm}

Combining Theorem \ref{t:CMV} and Theorem \ref{inclusion}, we extend Theorem \ref{t:CMV} in all dimensions $\geq6$.
An analog for self-mapping degree sets  also follows from the proof of Theorem \ref{inclusion}.

\begin{thm}\label{extension} 
For each positive integer $n\ne 1, 2, 4, 5$,
every finite set of integers containing $0$ is the mapping degree set 
of some  $n$-manifolds. 

For each positive integer $n\ne 1, 2$, every finite multiplicative set  containing $\{0,1 \}$
is the self-mapping degree set of some $n$-manifold. 
\end{thm}

We believe that Theorem \ref{extension} also holds for $n=4,5$, but we do not have a proof yet. Also, for the sake of completeness, note that the mapping degree sets of 1- and 2-dimensional manifolds are very special: they are either $\Z$ or $\{-k, -k+1, ... ,-1,0,1, ... , k-1, k\}$ (see \cite[Example 1.5]{NWW}).

Constructing specific non-realisable sets seems to be a more subtle problem:

\begin{que}\label{Q1}
Find a concrete subset of integers containing $0$ which is not a mapping degree set, and 
a concrete multiplicative  subset of integers containing $\{0, 1\}$ which is not a self-mapping degree set.
\end{que}

L\"oh and Uschold proved that each $D(M,N)$ is  recursively enumerable  \cite[Proposition A.1]{LU}. Theorem \ref{non1} and  \cite[Proposition A.1]{LU} 
together  imply that there must exist non-recursively enumerable multiplicative sets.
In Question \ref{Q1} we do hope to find a non-realisable subset of integers that can be written down explicitly without using any non-constructive existence result.

There are many examples of manifolds $M$ with $D(M)$  an infinite arithmetic progression (see  \cite{SWWZ} and its references). So far, 
we do not know any $D(M)$ which is an infinite geometric progression together with $\{0,1\}$. 
Parallel to Problems \ref{arith} and \ref{geom},  we ask the following  realisation question for (multiplicative) arithmetic and geometric progressions as self-mapping degree sets:

\

\begin{que}\label{Q2} \
\begin{itemize}
\item[(1)] Is every infinite arithmetic progression  together with $\{0,1\}$ a self-mapping degree set? 
\item[(2)] Is there an  infinite geometric progression together with $\{0,1\}$ which is a self-mapping degree set?
\end{itemize}
\end{que}

\section{Proofs}

\subsection{Proof of Theorem \ref{non1}}

Let $\mathcal P$ be the set of all prime numbers. For a subset $S\subseteq \mathcal P$, let 
$$\Pi(S)=\{p_1^{a_1}...p_k^{a_k} \ | \ p_1,...,p_k\in S,a_1,...,a_k\in\Z_{\ge0}\}\cup\{0, 1\}.$$

We have the following two elementary lemmas:

\begin{lem}\label{lem1}
$\Pi(S)$ is a multiplicative set.
\end{lem}

\begin{proof} Let $\alpha$, $\beta\in \Pi(S)$.

Case 1. $\alpha=0$ or $\beta=0$. Then $\alpha\beta=0\in \Pi(S)$.

Case 2. $\alpha,\beta\geq1$. Write 
$$\alpha=p_1^{a_1}...p_k^{a_k},\beta=q_1^{b_1}...q_l^{b_l}, \ \text{where} \ p_1,...,p_k,q_1,...,q_l\in S \ \text{and} \ a_1,...,a_k,b_1,...,b_l\in \mathbb{\Z}_{\ge 0}.$$ 
Then $\alpha\beta=p_1^{a_1}...p_k^{a_k}q_1^{b_1}...q_l^{b_l}\in \Pi(S)$.
\end{proof}

\begin{lem}\label{lem2} 
$\Pi(S)\cap \mathcal P=S$, for any subset $S\subseteq \mathcal P$.
\end{lem}

\begin{proof}  
In one direction, note that $S\subseteq\Pi(S)$, hence $S\subseteq \Pi(S)\cap \mathcal P$. 

Next, we show that $\Pi(S)\cap \mathcal P \subseteq S$: Let $p\in \Pi(S)\cap \mathcal P$. 
Since $p\in \mathcal P$, we have that $p\ne 0,1$. 
Then $p\in \Pi(S)$ means that $p= p_1^{a_1}...p_k^{a_k},$ where $p_1, ... , p_k\in S$ and $a_1,...,a_k\in\N$. But  
$p$ is a prime, and all $p_1, ... , p_k$ are primes,  hence $p=p_i$, for some $i=1,...,k$. Thus $p\in S$.
\end{proof}

Let now $\mathbb{P}(\mathcal P)$ be the set of all subsets of $\mathcal P$. By Lemma \ref{lem1}, there is  a map 
$$f\colon\mathbb{P}(\mathcal P)\to \{{\rm multiplicative\,\, subsets\,\, of}\,\, \mathbb{Z}\,\, {\rm containing}\,\, \{0,1\}\}$$
defined  by 
$$f(S)=\Pi(S)$$
for each $S\in \mathbb{P}(\mathcal P)$.
By Lemma \ref{lem2}, we have
  $$f(S)\cap \mathcal P=\Pi(S)\cap \mathcal P=S.$$ 
 Hence, $f$ is injective. Since $\mathbb{P}(P)$ is uncountable, we conclude that $f(\mathbb{P}(\mathcal P))$ is also uncountable. Thus, the set that consists of all multiplicative subsets of $\Z$ containing $\{0,1\}$ is uncountable. 
 
On the other hand, according to a theorem of M. Mather \cite{Ma}, there are only countably many homotopy classes of closed orientable $n$-manifolds.
It is easy to verify that if two closed orientable $n$-manifolds $M$ and $M'$ are homotopy equivalent, then $D(M)=D(M')$. 
So the subsets of $\Z$ which can be realised as sets of self mapping degrees 
$D(M)$ for some closed oriented $n$-manifolds $M$ are only countably many, for all positive integers $n$.   This completes the proof of Theorem \ref{non1}.

\subsection{Proof of Theorem \ref{inclusion}} 

We will prove the following precise version of Theorem \ref{inclusion}:

\begin{thm}\label{inclusionconcrete}
Let $M,N$ be closed oriented $n$-manifolds. For any integer $k\geq3$, there exist infinitely many  closed oriented hyperbolic $k$-manifolds $W$ such that $$D(M,N)=D(M\times W,N\times W).$$
\end{thm}

The proof of Theorem \ref{inclusionconcrete} is based on several results about closed oriented hyperbolic manifolds and the following specific form of \cite[Theorem 4.6]{NWW} (or \cite[Theorem 1.4]{Ne1} for self-maps):

\begin{thm}\label{t:product}
Let $M, N$ be two closed oriented manifolds of dimension $n$ and $W$ be a closed oriented manifold of dimension $k$. Suppose 
\begin{itemize}
\item[(i)] $W$ does not admit maps of non-zero degree from direct products $W_1\times W_2$, where $\dim W_1, \dim W_2 >0$ and
$\dim W_1+\dim W_2= k$, and
\item[(ii)]   for any  map $M\to W$, the induced homomorphism $H_k(M; \Q)\to H_k(W; \Q)$ is trivial.
\end{itemize}
Then $D(M\times W, N\times W)=D(M,N)\cdot D(W)$. In particular, if in addition
\begin{itemize}
\item[(iii)]  $D(W)=\{0,1\}$,
\end{itemize}
then $D(M\times W, N\times W)=D(M,N)$.
\end{thm}

Hence, in order to prove Theorem \ref{inclusionconcrete} we need to find hyperbolic manifolds $W$ satisfying (i)--(iii) in Theorem \ref{t:product}. First of all, hyperbolic manifolds satisfy (i), that is, they do not admit maps of non-zero degree from products~\cite{KL}. Next, given a closed oriented hyperbolic $k$-manifold $L$, with fundamental class $[L]\in H_k(L)$, its simplicial volume satisfies $\|L\|=\|[L]\|_1>0$  and  $|\deg(f)|\|L\|\le\|L\|$ for each map $f\colon L\to L$ \cite[6.1.4, 6.1.2]{Th}, thus $D(L)\subseteq\{-1,0,1\}$. (Here,  $\|\cdot\|_1$ denotes the $\ell^1$-semi-norm, which in top degree is the simplicial volume.)
Since $0,1$ always belong to $D(L)$,  we need $-1 \notin D(L)$. 
This is indeed the case quite often, as observed by S. Weinberger (see~\cite[Section 3]{Mu} and~\cite[Section 3.1]{Ne1}). 

To deal with (ii) and  (iii) of Theorem \ref{t:product},  we need the following two facts:

\begin{lem}\label{l:hyp}
For each $k\geq3$, there are infinitely many closed oriented hyperbolic $k$-manifolds $\{L_i\}$ such that $D(L_i)=\{0,1\}$ and the volume of $L_i$ is unbounded as $i$ tends to infinity.
\end{lem}
\begin{proof}
By a result of Belolipesky and Lubotszky \cite[Theorem 1.1]{BL}, for each $k\geq2$ and any finite group $\Gamma$, there exists a closed oriented hyperbolic $k$-manifold $L$ such that  $\mathrm{Isom}(L)\cong\Gamma$, where $\mathrm{Isom}(L)$ is the full isometry group of $L$. By  \cite[Theorem 6.4]{Th},  every map $f\colon L\to L$ of $|\deg(f)|=1$ is homotopic to an isometry when $k\geq3$. Note that each orientation reversing isometry must has even order. Hence, if	 $\Gamma$ is of odd order, then we have $\deg(f)=1$ and $D(L)=\{0,1\}$. 

Now let $L_i$ be a closed oriented hyperbolic $k$-manifold such that $\mathrm{Isom}(L_i)\cong\Z_{2i+1}$, the cyclic group of order 
$2i+1$. Then, the family $\{L_i\}$ contains  infinitely many hyperbolic $k$-manifolds with $D(L_i)=\{0,1\}$. The volume of $L_i$ is unbounded as $i$ tends to infinity. Indeed, if $k>3$, this follows directly from H. C. Wang's theorem \cite{Wa} that there are only finitely many hyperbolic $k$-manifolds with volume bounded by a fixed number $r>0$. If $k=3$, we have the following equality about hyperbolic volumes
$$V(L_i)=(2i+1)V(L_i/\Z_{2i+1}).$$
By a result of Meyerhoff \cite{Me}, the volumes of 3-dimensional hyperbolic orbifolds have a lower bound $C>0$, hence $V(L_i)> (2i+1)C$ is unbounded as $i$ tends to infinity.
\end{proof}

\begin{lem}\label{l:norm}
Let $M$ and $W$ be closed oriented manifolds of dimensions $n$ and $k$ respectively, and $\{\alpha_1,..., \alpha_n\}$ be a basis of $H_k(M;\mathbb{Q})$ such that 
\begin{itemize}
\item[(i)] each $\alpha_i$ is the image of a homology class in $H_k(M;\mathbb{Z})$, and 
\item[(ii)] $\|W\|>\max\{\|\alpha_i\|_1|i=1, ... , k\}.$ 
\end{itemize}
Then for any  map $M\to W$, the induced homomorphism $H_k(M; \Q)\to H_k(W; \Q)$ is trivial.
\end{lem}
\begin{proof}
Let $f\colon M\to W$ be a map.  Let $[W]$ be the fundamental class of $W$. Then for any $i$,
as integer homology classes, we have  $H_k(f)(\alpha_i)=d_i[W]$ for some $d_i\in\Z$.
By our assumption on $\|W\|$ and the functoriality of the $\ell^1$-semi-norm (cf.~\cite[p.8]{Gr}), we obtain
$$\|\alpha_i\|_1\geq\|H_k(f)(\alpha_i)\|_1=|d_i|\|W\|>|d_i|\|\alpha\|_1.$$
Thus $d_i=0$ holds and $H_k(f)(\alpha_i)=0$. Since $\{\alpha_1,\cdots,\alpha_n\}$ is a basis of $H_k(M;\mathbb{Q})$ and $H_k(f)$
is linear, $H_k(f)$ must be trivial. 
\end{proof}

Now we finish the proof of Theorem \ref{inclusionconcrete}:  Let $\{\alpha_1,..., \alpha_n\}$ be a basis of $H_k(M;\mathbb{Q})$ chosen as in Lemma \ref{l:norm} (i).
Recall that for each closed oriented hyperbolic manifold, its hyperbolic volume is proportional to the simplicial volume~\cite[Prop. 6.1.4]{Th}.  In Theorem \ref{t:product}, we can take $W$ to be some closed oriented hyperbolic $k$-manifold $L_i$ given by Lemma \ref{l:hyp} such that Lemma \ref{l:norm} (ii) is satisfied as well, that is, $$\|W\|>\max\{\|\alpha_i\|_1|i=1, ... , n\}.$$
Then $W$ satisfies conditions (i) (ii) and  (iii) of Theorem \ref{t:product} (by Lemma \ref{l:hyp}, Lemma \ref{l:norm} and~\cite{KL}). 
Finally, we note that there are infinitely many choices for $W$ (by Lemma \ref{l:hyp}), thus we conclude Theorem \ref{inclusionconcrete}.

\begin{rem}
Note that for $k>n$ (i.e. in all but finitely many dimensions for each $n$), condition (ii) of Theorem \ref{t:product} is automatically satisfied and thus Lemma \ref{l:norm} and the second part of Lemma \ref{l:hyp} about volumes is not needed in the proof of Theorem \ref{inclusionconcrete}.
\end{rem}

\subsection{Proof of Theorem \ref{extension}}
Let $A$ be a finite set of integers containing $0$. Then $A$ is the mapping degree set of some $3$-manifolds, by 
Theorem \ref{t:CMV}. For each $n=3+k$, where $k\geq3$, $A$ is also the mapping degree set of some $n$-manifolds, by Theorem \ref{inclusion}.

For the second part of Theorem \ref{extension}, note that there are only two finite multiplicative sets containing $\{0, 1\}$, namely $\{ 0, 1\}$ itself and $\{-1, 0, 1\}$. The first set is realised by a closed hyperbolic $n$-manifold $L$ as in Lemma \ref{l:hyp}. The second set is realised by $L\#\overline L$, where $\overline L$ is the manifold $L$ with the opposite orientation, since  $L\# \overline L$ admits a degree $-1$ self-map which is realised by the reflection about the $(n-1)$-sphere of the connected sum. Finally note that by \cite{Gr} we have 
$||L\#  \overline L||=||L||+||\overline L||=2||L||>0$. So $D(L\# \overline L)$ is finite.

\subsection{A proof for $D(\CP^n)=\{k^n \ | \  k\in \Z\}$ (Example \ref{cpn})} Without a precise reference for this known fact, we provide 
a proof: If $t$ is a generator of $H^2(\CP^n;\Z)$, then  
$H^*(\CP^n; \Z)=\Z[t]/(t^{n+1}).$
From this, one derives
that  $D(\CP^n)\subseteq \{k^n  \ | \  k\in \Z\}$. 
Below we  show that for each integer $k$,
there is a map $$f_k\colon\mathbb{CP}^n\rightarrow \mathbb{CP}^n$$ of degree $k^n$. We may assume that $k\ne 0$.

If $k>0$, define $$f_k[z_0:z_1:...:z_n]=[z_0^k:z_1^k:...:z_n^k], \ [z_0:z_1:...:z_n]\in\mathbb{CP}^n.$$
Note that for $\mathbb{CP}^1=\{[z_0:z_1:0:0:...:0]\}\subseteq \mathbb{CP}^n$, the map  $f_k|_{\mathbb{CP}^1}\colon\mathbb{CP}^1\rightarrow \mathbb{CP}^1$ has degree $k$. Since $[\mathbb{CP}^1]$ is a generator of $H_2(\mathbb{CP}^n;\mathbb{Z})$, we have that $f_k\colon H_2(\mathbb{CP}^n;\mathbb{Z})\rightarrow H_2(\mathbb{CP}^n;\mathbb{Z})$ is given by multiplication by $k$, and thus, by algebraic duality, $f_k^*\colon H^2(\mathbb{CP}^n;\mathbb{Z})\rightarrow H^2(\mathbb{CP}^n;\mathbb{Z})$ is given by $f_k^*(t)=kt$. By the ring structure of $H^*(\CP^n; \Z)$, we have $\deg(f_k)=k^n$.

If $k<0$, define
$$ f_k[z_0:z_1:...:z_n]=[\bar{z_0}^{-k}:\bar{z_1}^{-k}:...:\bar{z_n}^{-k}].$$
The map $f_k|_{\mathbb{CP}^1}\colon\mathbb{CP}^1\rightarrow\mathbb{CP}^1$ has still degree $k$, and thus, as above, 
we have $\deg(f_k)=k^n$.

\subsection*{Acknowledgements} The authors thank Clara L\"oh and Antonio Viruel for their enlightening comments on their results in \cite{LU} and \cite{CMV} respectively. H.S. is partially supported by Simons Collaboration Grant 615229. C.N. would like to thank IHES for the hospitality.

\bibliographystyle{ams}

\end{document}